\documentclass[12pt]{amsart}
\usepackage{float} 
\usepackage{amsmath,amsthm,amssymb,verbatim,stmaryrd,xcolor,microtype,graphicx,aliascnt,needspace,mathtools,booktabs} 
\usepackage[shortlabels]{enumitem}
\usepackage[T1]{fontenc}
\usepackage[utf8]{inputenc}
\usepackage[english]{babel} % allow for hyphenation for words with dash: e.g. ``non-archimedean'' does not break, but ``non\hyph archimedean'' does
\usepackage[top=3.2cm,bottom=3.2cm,left=3.5cm,right=3.5cm]{geometry}
\usepackage[bookmarksopen=true,bookmarksopenlevel=2,bookmarksdepth=2,linktoc=page,colorlinks,linkcolor={red!70!black},citecolor={red!80!black},urlcolor={blue!80!black},pdfpagemode=UseOutlines,pdfstartview={fitH}]{hyperref}
\usepackage{longtable}
\usepackage[utf8]{inputenc}
\usepackage{adjustbox}
\usepackage{caption}
\usepackage{latexsym}

\usepackage{tikz}
\usepackage{tikz-cd}
\usepackage[mathcal]{euscript}      % use the euscript caligraphic font with \mathscr{...}

\usepackage[colorinlistoftodos,bordercolor=orange,backgroundcolor=orange!20,linecolor=orange,textsize=footnotesize]{todonotes} \makeatletter \providecommand \@dotsep{5} \def\listtodoname{List of Todos} \def\listoftodos{\@starttoc{tdo}\listtodoname} \makeatother %\Todo{} for margin notes, suppress in pdf with option [disable]

\usepackage{newtxtext,newtxmath}                              % times font
\usepackage{bbm}

\allowdisplaybreaks % allows equation environments to break between the pages

\usepackage{etoolbox}
\makeatletter
\patchcmd{\@startsection}{\@afterindenttrue}{\@afterindentfalse}{}{}             %omit indentation of the first paragraph of a section
\patchcmd{\part}{\bfseries}{\bfseries\LARGE}{}{}
\patchcmd{\section}{\scshape}{\bfseries}{}{}\renewcommand{\@secnumfont}{\bfseries} %boldface no smallcaps section and subsection titles with numbers
\patchcmd{\@settitle}{\uppercasenonmath\@title}{\large}{}{}
\patchcmd{\@setauthors}{\MakeUppercase}{}{}{}
\addto{\captionsenglish}{} %boldface no smallcaps Abstract
\addto{\captionsenglish}{} %boldface no smallcaps Figure
\addto{\captionsenglish}{} %boldface no smallcaps Table
\makeatother

\usepackage{fancyhdr}

\addto\extrasenglish{}  % Use capitalized ``Section'' (also for (sub)subsections
\theoremstyle{plain}
\newtheorem{thm}{Theorem}[section] % provides command \autoref{}, which produces citations like ``Theorem 1.1''.
\newaliascnt{lemma}{thm}\newtheorem{lemma}[lemma]{Lemma}\aliascntresetthe{lemma}
\newaliascnt{cor}{thm}\newtheorem{cor}[cor]{Corollary}\aliascntresetthe{cor}
\newaliascnt{prop}{thm}\newtheorem{prop}[prop]{Proposition}\aliascntresetthe{prop}
\newaliascnt{claim}{thm}\aliascntresetthe{claim}
\newtheorem*{claim*}{Claim}
\newtheorem*{thm*}{Theorem}
\newtheorem*{lem*}{Lemma}
\newtheorem*{cor*}{Corollary}

\theoremstyle{definition}
\newaliascnt{df}{thm}\newtheorem{df}[df]{Definition}\aliascntresetthe{df}
\newaliascnt{rem}{thm}\newtheorem{rem}[rem]{Remark}\aliascntresetthe{rem}
\newaliascnt{ex}{thm}\newtheorem{ex}[ex]{Example}\aliascntresetthe{ex}
\newaliascnt{conj}{thm}\aliascntresetthe{conj}
\newaliascnt{problem}{thm}\aliascntresetthe{problem}

\newtheorem*{df*}{Definition}
\newtheorem*{ex*}{Example}
\newtheorem*{rem*}{Remark}

\DeclareRobustCommand{\gobblefour}[5]{}    % Command \SkipTocEntry for suppressing a section title in TOC
\DeclareMathOperator{\yo}{{h}}
\DeclareMathOperator{\Hom}{Hom}

\DeclareMathOperator{\Aff}{Aff}
\DeclareMathOperator{\Pres}{Pres}
\DeclareMathOperator{\Sch}{Sch}
\DeclareMathOperator{\SAff}{SAff}
\DeclareMathOperator{\SSch}{SSch}
\DeclareMathOperator{\Mf}{Mf}

\DeclareMathOperator{\ASC}{ASC}

\DeclareMathOperator{\Spec}{Spec}
\DeclareMathOperator{\Cong}{Cong}
\DeclareMathOperator{\PSpec}{PSpec}
\DeclareMathOperator{\Sh}{Sh}
\DeclareMathOperator{\PrSh}{PrSh}

\DeclareMathOperator{\colim}{colim\,}
\DeclareMathOperator{\Sets}{{Sets}}
\DeclareMathOperator{\TopSp}{{Top}}

\DeclareMathOperator{\Rings}{{Rings}}
\DeclareMathOperator{\SRings}{{SRings}}

\DeclareMathOperator{\Alg}{{Alg}}

\DeclareMathOperator{\Loc}{{Loc}}

\DeclareMathOperator{\res}{{res}}

\newcommand{\A}{{\mathbb A}}

\newcommand{\C}{{\mathbb C}}

\newcommand{\N}{{\mathbb N}}

\newcommand{\R}{{\mathbb R}}

\newcommand{\Z}{{\mathbb Z}}

\newcommand{\cB}{{\mathcal B}}
\newcommand{\cC}{{\mathcal C}}

\newcommand{\cF}{{\mathcal F}}

\newcommand{\cI}{{\mathcal I}}

\newcommand{\cO}{{\mathcal O}}
\newcommand{\cP}{{\mathcal P}}

\newcommand{\cS}{{\mathcal S}}
\newcommand{\cT}{{\mathcal T}}
\newcommand{\cU}{{\mathcal U}}
\newcommand{\cV}{{\mathcal V}}
\newcommand{\cW}{{\mathcal W}}
\newcommand{\cX}{{\mathcal X}}

\newcommand{\fc}{{\mathfrak c}}

\newcommand{\fp}{{\mathfrak p}}

\newcommand{\fs}{{\mathfrak s}}

\newcommand{\0}{{\mathbf{0}}}
\newcommand{\1}{{\mathbf{1}}}
\newcommand{\bbone}{{\mathbbm{1}}}
\renewcommand{\phi}{{\varphi}}

\newcommand{\im}{\textup{im}}

\newcommand{\can}{\textup{can}}

\renewcommand{\top}{\textup{top}}
\newcommand{\loc}{\textup{loc}}
\newcommand{\op}{\textup{op}}

\renewcommand{\leq}{\leqslant}
\newcommand{\gen}[1]{\langle #1 \rangle}

\renewcommand{\emptyset}{\varnothing}

\DeclareMathOperator{\Op}{Open}

\newcommand{\Grot}{\mathcal{T} }              %Grothendieck topology
\newcommand{\bset}[1]{\{ #1  \}}

\title{Toolkit for the algebraic geometer\\[10pt] \normalsize Lecture notes}
\author{Sourayan Banerjee}
\address{\rm Sourayan Banerjee, Indian Institute of Technology Kanpur, India }
\email{sourayan244@gmail.com}

\author{Oliver Lorscheid}
\address{\rm Oliver Lorscheid, University of Groningen, the Netherlands}
\email{o.lorscheid@rug.nl}

\author{Alejandro Mart\'inez M\'endez}
\address{\rm Alejandro Mart\'inez M\'endez, University of Groningen, the Netherlands}
\email{a.m.m.martinez.mendez@rug.nl}

\author{Alejandro Vargas}
\address{\rm Alejandro Vargas, University of Warwick, United Kingdom. }
\email{alejandro@vargas.page}

\hypersetup{pdftitle={Toolkit for the algebraic geometer},pdfauthor={Sourayan Banerjee, Oliver Lorscheid, Alejandro Mart\'inez M\'endez and Alejandro Vargas}}

\begin{document}

\begin{abstract}
 In this text, we outline a theory of schemes associated with a site, which generalizes a variety of geometries, such as manifolds, schemes, analytic spaces, simplicial complexes, and more. We present an abstract process of gluing model spaces via sheaf theory and recover \emph{a posteriori} the underlying topological spaces that are often present in the construction of such geometric objects. We apply this formalism to semiring schemes and reason why the usual definition of semiring schemes has to be considered as the good approach to the geometry of semirings.
\end{abstract}

\maketitle

\begin{small} \tableofcontents \end{small}

%=======================================================
\section*{Introduction}
%\label{sec:introduction}
%=======================================================

In this text, we exhibit the mechanism behind the construction of geometric objects from a class of model spaces, such as manifolds, which are formed from open subsets in Euclidean space, and schemes, which are formed from spectra of rings. 

A necessary input for gluing model spaces along open subspaces is not only a notion of open subobjects, but
also a notion of coverings by these open subobjects. This information is accurately captured by Grothendieck pretopologies. It turns out that a Grothendieck pretopology is sufficient for the construction of geometric objects that look locally like our model spaces.

%=======================================================
\subsection*{Three perspectives on geometry}
%=======================================================

There are three prominent approaches to construct geometric objects from model spaces, which we review in the examples of smooth manifolds and schemes in the following; cf.\ \autoref{table:PointsOfView} for an overview.

%=======================================================
\subsubsection*{Definition in terms of charts}
%=======================================================

A smooth manifold is typically defined as a (second countable) topological space $X$ together with an equivalence class of a smooth atlas, which is a collection of charts $\varphi_i:U_i\to X$ from open subsets $U_i$ of $\R^n$ for which the transition maps $\varphi^{-1}_j\circ\varphi_i$ are smooth homeomorphisms.

%=======================================================
\subsubsection*{Definition in terms of a structure sheaf}
%=======================================================

A scheme is typically defined as a topological space $X$ together with a structure sheaf $\cO_X$ in $\Rings$ such that every point has a neighbourhood $U$ that is isomorphic to the spectrum of $\cO_X(U)$. 

%=======================================================
\subsubsection*{Definition as a functor of points}
%=======================================================

A well-known fact in algebraic geometry is that the category of schemes can be embedded into the functor category from rings to sets by identifying a scheme $X$ with the functor of points $\Hom(\Spec(-),X)$. It is possible to characterize schemes from the perspective of functors of points; cf.\ \cite[4.1]{GortzWedhorn2010} and~\cite{ToenVaquie2009}.

%=======================================================
\begin{table}[t]
 \caption{Three definitions of manifolds and schemes}
 \label{table:PointsOfView}
 \begin{tabular}{lp{5.5cm}p{5.5cm}}
  \toprule
  Definition & Smooth manifolds  &  $\text{Schemes}$ \\
  \midrule
%------------------------Definition A----------------------------
  \parbox{1.8cm}{Topological space with charts} &
  \parbox{5.5cm}{$X\in\TopSp$ with an open cover by charts $\varphi_i:U_i\to X$ with $U_i\subset\R^n$ open such that $\varphi_j^{-1} \circ \varphi_i$ is smooth}  &
  \parbox{5.5cm}{$X\in\TopSp$ with a cover by open immersions $\iota_i : U_i \to X$ such that $\iota_j^{-1} \circ \iota_i$ is a scheme morphism}  \cr \\
%------------------------Definition B----------------------------
  \parbox{1.8cm}{Space with structure sheaf} &
  \parbox{5.5cm}{$X\in\TopSp$ with a structure sheaf $\cO_X(-) = C^\infty(-, \mathbb R)$ of smooth functions}  &
  \parbox{5.5cm}{$X\in\TopSp$ with a structure sheaf $\cO_X(-)=\Hom(-,\A^1_\Z)$ of regular functions}  \cr \\
%------------------------Definition C----------------------------
  \parbox{1.5cm}{Functor of points} &
  \parbox{5.5cm}{A functor from open subsets of $\R^n$ to $\Sets$ of the form $C^\infty(-,X)$}  &
  \parbox{5.5cm}{A functor $h_X : \Rings \to \Sets$ of the form $R \mapsto X(R)$}  \cr \\
\toprule
\end{tabular}
\end{table}
%=======================================================

One can pass forth and back between these definitions: a scheme can be considered as a topological space together with an equivalence class of atlases, i.e., coverings by affine schemes. A smooth manifold can be considered equivalently as a topological space together with a sheaf of smooth functions, and this viewpoint is typically used for the generalization of smooth manifolds to analytic spaces. Furthermore, a smooth manifold $X$ can be identified with the functor of points $\cC^\infty(-,X)$ on the category of open subsets of Euclidean space $\R^n$ and smooth maps. See \cite[Ch.\ 1]{eh06} for an exposition of these different viewpoints on schemes.

%=======================================================
\subsection*{Top down approach to \texorpdfstring{$\fs$}{s}-schemes}
%=======================================================
In the generality of a (possibly abstract) category $\Aff$ of model spaces together with a Grothendieck pretopology $\cT$, we do not know \emph{a priori} what the underlying topological space of objects in $\Aff$ might be. The reader might want to think of the category $\Aff=\SRings^\op$ of affine semiring schemes: it is not clear from the outset what is the good definition of the prime spectrum of a semiring as a topological space (cf.\ \cite{GKMX}, which provides two different constructions that are based on prime ideals and prime $k$-ideals, respectively).

This forces us to take the third route through functors of points, which are sheaves on the site $\fs=(\Aff,\cT)$. Intuitively speaking, $\cT$ provides us with a notion of open subfunctors. We define an $\fs$-scheme as the colimit sheaf of a ``monodromy free'' system of such open subfunctors.

This provides a rigorous definition of the category $\Sch_\fs$ of $\fs$-schemes and recovers many categories of geometric objects, such as manifolds, schemes, analytic spaces, simplicial complexes, and more.

%=======================================================
\subsection*{Points for \texorpdfstring{$\fs$}{s}-schemes}
%=======================================================
\emph{A posteriori}, we can assign topological spaces to the objects $X$ in $\Aff$ as follows. The Grothendieck pretopology $\cT$ provides a notion of open subschemes of $X$. This collection of open subschemes forms a locale, and a result in \cite{toolkit} shows that this locale coincides with the collection of open subsets of an associated topological space $\underline X$ via Stone duality.

If $\Aff$ is anti-equivalent to a category $\Alg$ of algebraic objects, e.g.\ (semi)rings, then $\underline X$ comes equipped with a sheaf in $\Alg$. This shows that the underlying topological space of an object $X$ is intrinsically determined by the Grothendieck pretopology $\cT$.

In many interesting examples, such as manifold or schemes, the underlying topological space $\underline X$ turns out to coincide with the given topological space, i.e., the open subset of $\R^n$ in the case of manifolds and the prime spectrum of a ring in the case of an affine scheme $\Spec R$. In the case of semiring schemes, $\underline X$ coincides with the space of all prime ideals; see \cite{L26} for a proof.

%=======================================================
\subsection*{Relation to other approaches}
%=======================================================
This text is based on the forthcoming paper \cite{toolkit} and serves as a reader friendly outlook to this theory. The ideas of this text are not entirely new, but bear strong similarities with \cite[section 2.2]{Toen-Vaquie08} and \cite{lurie2009derivedalgebraicgeometryv}. The main novelty of our approach is that it is stripped from all unnecessary technical conditions that are present in the other approaches. Moreover, we continue the study of our theory beyond the point that this is done in \cite{Toen-Vaquie08} and \cite{lurie2009derivedalgebraicgeometryv}.

%=======================================================
\subsection*{Content outline}
%=======================================================
This text is organized as follows. In \autoref{section: RecapOnSheafTheory}, we provide a summary of the necessary concepts from sheaf theory, such as Grothendieck pretopologies, sheaves on a site and the Yoneda embedding.

In \autoref{section: s-schemes}, we introduce $\fs$-presentations and $\fs$-schemes. We provide some fundamental facts for this theory and conclude with proof outlines for the facts that smooth manifolds and schemes can be considered as particular instances of $\fs$-schemes, as indicated in this introduction.

In \autoref{section: PointsInSSchemes}, we review Stone duality and apply it to the locale of open subschemes, which equips an affine $\fs$-scheme with an underlying topological space $\underline X$. We show that this recovers open subsets of $\R^n$ and affine schemes as topological spaces.

In \autoref{section: further examples}, we discuss further examples of $\fs$-scchemes, such as various types of manifolds and simplicial complexes.

In \autoref{section: semiring schemes}, we apply our theory to semiring schemes, which provides an alternative perspective on semiring schemes (compared to \cite{GKMX}). We explain the links between these two approaches (with proofs in \cite{L26}). Finally, we explain how to build other topological spaces (of prime $k$-ideals and prime congruences) for semiring schemes, in a way that avoids a loss of information due to Grothendieck topologies for which the Yoneda functor fails to be an embedding of categories.

%=======================================================
\subsection*{Disclaimer}
%=======================================================

In this text, we ignore any set theoretic issue related to categories of functors, which can be solved by standard methods, such as Grothendieck universes or hierarchies of classes.

%=======================================================
\subsection*{Acknowledgements}
%=======================================================
The authors thank the organizers Marc Masdeu and Joaquim Roe of the Barcelona workshop on the Geometry of Semirings for their hopsitality and for making this event happening. AV is supported by the UK Engineering and Physical Sciences Research Council (EPSRC), grant number EP/X02752X/1.

%%%%%%%%%%%%%%%%%%%%%%%%%%%%%%%%%%%%%%%%%%%%%%%%%%%%%%%%%%%%%%%%%%%%%%%%%%%%%%%%%%%%%%%%%%%%%%
%%%%%%%%%%%%%%%%%%%%%%%%%%%%%%%%%%%%%%%%%%%%%%%%%%%%%%%%%%%%%%%%%%%%%%%%%%%%%%%%%%%%%%%%%%%%%%

%=======================================================
\section{Recap on sheaf theory}
\label{section: RecapOnSheafTheory}
%=======================================================

Let $\Aff$ be a category, which we intend to interpret as a category of affine schemes. The reader might want to think of one of our running examples: the category $\Aff_{\cC^\infty}$ of ``smooth'' open subsets of Euclidean space (equipped with smooth maps) or the opposite category $\Aff_\Z$ of $\Rings$.

In order to pass from affine schemes to more general schemes, we require the notions of covering families and sheaves. In this section, we review the necessary concepts.

%=======================================================
\subsection{Sites}
%=======================================================

A \emph{Grothendieck pretopology} $\Grot$ on $\Aff$ is an assignment to each object $X\in\Aff$ of a collection of families $\{\iota_i:U_i\to X\}$ of morphisms in $\Aff$, that satisfies the following properties:
\begin{enumerate}
 \item \emph{Isomorphisms:} $\bset{\varphi : Y \to X}$ is in $\cT$ for every isomorphism $\varphi:Y\to X$.
 \item \emph{Local character:} if $\bset{\iota_i:U_i\to X}_{i \in I}$ and $\bset{\iota_{ij}:U_{ij} \to  U_j}_{j \in J_i}$ (one family for each $i\in I$) are in~$\cT$, then also $\bset{ \iota_i\circ\iota_{ij}:U_{ij}\to X}_{i \in I, j \in J_i}$ is in $\cT$.
 \item \emph{Base change:} if $\bset{\iota_i:U_i\to X}_{i\in I}$ is in $\cT$ and $f:Y\to X$ is a morphism in $\Aff$, then $\Aff$ contains the fibre product $U_i \times_X Y$ for all $i\in I$ and the family of canonical projections $\{U_i \times_X Y \to Y\}$ is in $\cT$.
\end{enumerate}
We call the families in $\cT$ the \emph{covering families} of $\cT$. A \emph{site} is a pair $\fs:=(\Aff, \Grot)$ of a category $\Aff$ together with a Grothendieck pretopology $\Grot$ for $\Aff$.

In the following examples, we call a family $\{\iota_i:U_i\to X\}$ of maps \emph{jointly surjective} if every point of $X$ is contained in the image of some $\iota_i$. This is a condition that holds for the covering families of many prominent Grothendieck pretopologies.

\begin{ex}\label{ex: Grothendieck topologies from topological spaces}
 Grothendieck pretopologies generalize usual topologies in the following way: Let $X$ be a topological space, $\cB$ a basis for the topology of $X$ that is closed under finite intersections and $\Op(X)$ the category of all open subsets of $X$ with the inclusion maps as morphisms. Then the collection of all jointly surjective families $\{\iota_i:U_i\to V\}$ with $U_i,V\in\cB$ is a Grothendieck pretopology for $\Op(X)$.
\end{ex}

\begin{ex}\label{ex: Euclidean topology}
 The \emph{Euclidean topology} for the category $\Aff_{\cC^\infty}$ of ``smooth'' open subsets in Euclidean space is the Grothendieck pretopology that consists of jointly surjective families $\{\iota_i:U_i\to X\}$ of smooth open topological embeddings.
\end{ex}

\begin{ex}\label{ex: Zariski topology}
 The \emph{Zariski topology} for the category $\Aff_{\Z}$ of usual affine schemes is the Grothendieck pretopology that consists of jointly surjective families $\{\iota_i:U_i\to X\}$ of open immersions.
\end{ex}

\begin{ex}\label{ex: principal Zariski topology}
 A variation of the Zariski topology is the \emph{principal Zariski topology} that consists of jointly surjective families $\{\iota_i:U_i\to X\}$ of \emph{principal open immersions}, which are open immersions for which the induced maps $\Gamma X\to \Gamma U_i$ between the respective coordinate rings is a localization of $\Gamma X$ at an element $h_i$. Our generalization to semiring schemes in \autoref{section: semiring schemes} is based on the principal Zariski topology.
\end{ex}

%=======================================================
\subsection{Sheaves}
\label{sub:Sheaves}
%=======================================================

A \emph{presheaf on $\Aff$} is a functor $\cF : \Aff^{\text{op}} \to \Sets$ and a morphism of presheaves is a natural transformation. We denote the category of presheaves on $\Aff$ by $\PrSh_{\Aff}$.

Let $\cS=\{\iota_i:U_i\to X\}$ be a covering family in $\cT$. Given an element $f\in \cF(X)$, we denote its pullback along $\iota_i:U_i\to X$ as $f_i=\iota_i^\ast(f)$. We denote by $f_{i,j}=\res_{U_i,U_{ij}}(f_i)$ the pullback of $f_i\in\cF(U_i)$ from $U_i$ to the fibre product $U_{ij}=U_i\times_XU_j$. Note that if $f_i\in\cF(U_i)$ and $f_j\in \cF(U_j)$ comes from an element $f\in\cF(X)$, then $f_{i,j}=f_{j,i}$. This establishes a map
\[
 \begin{array}{crcl}
  \Psi_\cS: & \cF(X) & \longrightarrow & \big\{ (f_i)\in\prod\cF(U_i) \, \big| \, f_{i,j}=f_{j,i} \big\} \\
                & f      & \longmapsto     & (\iota_i^\ast(f)).
 \end{array}
\]

A \emph{sheaf on $\fs=(\Aff,\cT)$} is a presheaf $\cF$ on $\Aff$ that satisfies the \emph{sheaf axiom for $\cT$}: the map $\Psi_\cS$ is a bijection for every covering family $\cS$ in $\cT$. A \emph{morphism of sheaves on $\fs$} is a morphism of presheaves, which defines the category $\Sh_\fs$ of sheaves on $\fs$ as a full subcategory of $\PrSh_{\Aff}$. 

This embedding comes with a left adjoint $(-)^\#:\PrSh_{\Aff}\to\Sh_\fs$, which sends a presheaf $\cF$to its \emph{sheafification} $\cF^\#$.

%=======================================================
\subsection{The Yoneda embedding}
\label{sub:Yoneda}
%=======================================================

Let $\fs=(\Aff,\cT)$ be a site. 
A presheaf $\cF$ is \emph{representable} if it is in the essential image of the embedding $\tilde \yo:\Aff\to\PrSh_{\Aff}$, given by $\tilde \yo_Y=\Hom(-,Y)$. The \emph{Yoneda functor} $\yo:\Aff\to\Sh_\fs$ is the composition of $\tilde\yo$ with the sheafification functor $(-)^\#:\PrSh_{\Aff}\to\Sh_\fs$.

We call a family $\cS=\{\varphi:U_i\to X\}$ \emph{subcanonical} if $\Hom(-,Y)$ satisfies the sheaf axiom with respect to $\cS$ for every $Y$ in $\Aff$, i.e., $\Psi_\cS$ is a bijection for every representable presheaf $\cF=\Hom(-,Y)$. A Grothendieck pretopology $\cT$ is \emph{subcanonical} if all of its covering families are subcanonical. The \emph{canonical topology $\cT_\can$ on $\Aff$} consists of all subcanonical families $\cS$ and is the finest subcanonical Grothendieck pretopology on $\Aff$.

In other words, the following are equivalent:
\begin{enumerate}
 \item $\cT$ is subcanonical;
 \item $\cT$ is contained in $\cT_\can$;
 \item $\Hom(-,Y)$ is a sheaf on $\fs$ for all $Y$ in $\Aff$;
 \item the Yoneda functor $\yo:\Aff\to\Sh_\fs$ is fully faithful.
\end{enumerate}
If $\cT$ is subcanonical, we can thus consider objects $Y$ of $\Aff$ as sheaves $\yo_Y$ on $\fs=(\Aff,\cT)$.

\bigskip

\begin{center}
 {\bf\textit{Throughout the rest of the text, we assume that $\cT$ is subcanonical.}}
\end{center}

%%%%%%%%%%%%%%%%%%%%%%%%%%%%%%%%%%%%%%%%%%%%%%%%%%%%%%%%%%%%%%%%%%%%%%%%%%%%%%%%%%%%%%%%%%%%%%
%%%%%%%%%%%%%%%%%%%%%%%%%%%%%%%%%%%%%%%%%%%%%%%%%%%%%%%%%%%%%%%%%%%%%%%%%%%%%%%%%%%%%%%%%%%%%%

%=======================================================
\section{\texorpdfstring{$\fs$}{s}-Schemes}
\label{section: s-schemes}
%=======================================================

Let $\fs=(\Aff,\cT)$ be a site. An $\fs$-scheme is a certain type of colimit of representable sheaves $h_Y=\Hom(-,Y)$ on $\Aff$. For simplicity, we assume that $\Aff$ is complete and cocomplete.

%=======================================================
\subsection{Principal opens}
\label{sub:PrincipalOpens}
%=======================================================

A \emph{principal open} is a morphism in $\Aff$ that appears in a covering family of $\cT$. We denote the class of all principal opens by $\cP=\cP_\cT$. The axioms of a Grothendieck pretopology imply the following properties for $\cP$:
\begin{enumerate}[label=(P\arabic*)]
 \item \label{PO1} all isomorphisms are in~$\cP$;
 \item \label{PO2} $\cP$ is closed under composition;% if both $Z \to Y$ and $Y \to X$ are in $\cP$, then  $Z \to Y \to X$ is in $\cP$.
 \item \label{PO3} $\cP$ is stable under base change: if $f: U \rightarrow X$ is in $\cP$ and $g: Y \rightarrow X$ is any morphism in $\Aff$, then the canonical projection $U \times_{X}Y \rightarrow Y$ is in $\cP$.
\end{enumerate}
Conversely, given a class of morphisms $\cP$, we define $\gen{\cP}_\can$ as the finest subcanonical Grothendieck topology for which all principal opens are in~$\cP$. The relation between $\cP$ and $\cT$ is as follows.

\begin{lemma}
 Let $\cP$ be a class of morphisms and $\cT=\gen{\cP}_\can$. Then $\cP_\cT\subset\cP$, with equality if and only if $\cP$ satisfies properties \ref{PO1}--\ref{PO3}. Conversely, a Grothendieck pretopology $\cT$ is contained in $\cT'=\gen{\cP_\cT}_\can$ and $\cP_{\cT'}=\cP_\cT$.
\end{lemma}

We call a Grothendieck pretopology $\cT$ \emph{subcanonically saturated} if $\cT=\gen{\cP_\cT}_\can$. Most interesting Grothendieck pretopologies are subcanonically saturated. For exceptions, see \autoref{ex: topology that is not subcanonically saturated} and \autoref{rem: subcanonically saturated Grothendieck topology for finite sets}.

\begin{rem}
 The term ``principal opens'' refers to the fact that in scheme theory and its generalizations to semirings (cf.\ \autoref{section: semiring schemes}) and other algebraic structures, it is easier to work with principal open immersions than with arbitrary open immersions. 
 
 This is already visible in the theory of usual schemes: in contrast to principal open immersions, which correspond to localizations of rings at a single element, a purely algebraic characterization of open immersions of affine schemes in terms of their coordinate algebras is a difficult challenge; also cf.\ \autoref{ex: principal Zariski topology}.

 It is a more profound problem for categories $\Aff$ for which it is not clear \emph{a priori} what the underlying topological space of an object would be; see \autoref{section: PointsInSSchemes} for more details.
\end{rem}

\begin{ex}\label{ex: topology that is not subcanonically saturated}
 Let $\Aff=\Rings^\op$ be the category of (usual) affine schemes and $\cT$ the Grothendieck pretopology that consists of all covering families $\{U_i\to X\}$ of principal open immersions for which $U_i\to X$ is an isomorphism for some $i$. Then $\cP_\cT$ is the class of all principal open immersions and $\gen{\cP_\cT}_\can$ consists of all Zariski coverings $\{U_i\to X\}$ by principal open immersions, which do not need to contain an isomorphism in general. Thus the inclusion $\cT\subset \gen{\cP_\cT}_\can$ is not an equality, i.e., $\cT$ is not subcanonically saturated.
\end{ex}

%=======================================================
\subsubsection{Principal opens for the Euclidean topology}
\label{subsubsection: principal opens for the Euclidean topology}
%=======================================================

Let $\Aff=\Aff_{\cC^\infty}$ be the category of open subsets of Euclidean space together with smooth maps, equipped with the Euclidean topology $\cT$ from \autoref{ex: Euclidean topology}, which consists of jointly surjective families $\{U_i\to X\}$ of smooth open topological embeddings. In this case, $\cP=\cP_\cT$ is the class of smooth open topological embeddings, and $\cT$ is subcanonically generated, as can be seen as follows.

The Euclidean topology $\cT$ is subcanonical: the functor $\Hom(-,Y)$ satisfies the sheaf axiom with respect to covering families $\{U_i\to X\}$ in $\cT$ because a smooth map $\varphi:X\to Y$ corresponds a family of smooth maps $\varphi_i:U_i\to Y$ that coincide pairwise on the ``intersection'' $U_i\times_X U_j$.

On the other hand, let $\cS=\{\iota_i:U_i\to X\}$ be a subcanonical family of smooth open topological embeddings. We claim that $\cS$ is in $\cT$, i.e., $\cS$ is jointly surjective. Indded, if this were not the case, i.e., if there was a point $x\in X-\bigcup\im\iota_i$, then also the base change $\{U_i\times_X\{x\}\to \{x\}\}$ would be subcanonical (using the base change property of the canonical topology). 
Since $x\notin\im\iota_i$, we have $U_i\times_X\{x\}=\emptyset$ for all $i$. 
Thus, for any $Y$ in $\Aff$, there is a unique family of morphisms $U_i\times_X\{x\}\to Y$ that coincides pairwise on the (empty) intersections, but there is in general more than one smooth map~$\{x\}\to Y$. 
This is the desired contradiction and shows that $\cS$ is in $\cT$.

%=======================================================
\subsubsection{Principal opens for the Zariski topology}
\label{subsubsection: principal opens for the Zariski topology}
%=======================================================

Let $\Aff=\Aff_\Z$ be the category of affine schemes, equipped with the Zariski topology $\cT$ from \autoref{ex: Zariski topology}, which consists of jointly surjective families $\{U_i\to X\}$ of open immersions. In this case, $\cP=\cP_\cT$ is the class of open immersions, and $\cT$ is subcanonically generated, which can be justified in a similar way as for the Euclidean topology; see \autoref{subsubsection: principal opens for the Euclidean topology}. In particular, the proof that the Euclidean topology is subcanonical carries over \emph{verbatim} to the Zariski topology.

The proof that every subcanonical family $\cS=\{U_i\to X\}$ of open immersions is in $\cT$ requires a slight adaptation. Assume that this was not the case, i.e., there was a point $x\in X-\bigcup\im\iota_i$. Then let $Z=\Spec k(x)$ be the spectrum of the residue field of $x$ and $Z\to X$ the canonical inclusion, which has image $x$. Pulling back the family $\cS$ to $Z$ yields the subcanonical family $\{U_i\times_XZ\to Z\}$, and $U_i\times_XZ$ is the empty scheme for all~$i$. This leads to the desired contradiction of the sheaf axiom for a suitably chosen affine scheme $Y$ (e.g. for $Y=Z\sqcup Z$).

%=======================================================
\subsection{\texorpdfstring{$\fs$}{s}-Presentations}
%=======================================================

Let $\fs=(\Aff,\cT)$ be a site, $\cP$ the class of its principal opens and $\Sh_\fs$ the category of sheaves on $\fs$. An \emph{$\fs$-presentation} is a certain type of diagram $\cU$ in $\Aff$ that characterizes an $\fs$-scheme $X$ as its colimit in $\Sh_\fs$. We introduce all required notions in the following.

A \emph{diagram in $\Aff$} is a functor $\cU:\cI\to\Aff$ from an ``index'' category $\cI$ to $\Aff$. We typically write $U_i=\cU(i)$ for the image of objects $i$ of $\cI$. We say that $\cU$ is a diagram \emph{of principal opens} if all morphisms $\cU(i\to j)$ are in~$\cP$. A diagram $\cU$ is \emph{monodromy free} if the canonical inclusion $U_i\to\colim_{\Sh_\fs}\cU$ is a monomorphism of sheaves on $\fs$ for every $i$ in $\cI$.

\begin{df}
 An \emph{$\fs$-presentation} is a monodromy free diagram $\cU$ of principal opens in $\Aff$.
\end{df}

\begin{rem}\label{rem: intrinsic characterization of monodromy free diagrams}
 Monodromy-free diagrams can be characterized intrinsically \emph{without reference to $\Sh_\fs$}, which is important for certain proofs. Even though this intrinsic characterization does make an explicit appearance in this text, we include a description for the sake of its relevance.

 A \emph{path in $\cU$} is a subdiagram $\Pi$ of the form
 \[
  U_0 \ \stackrel{\alpha_1}\longleftrightarrow \ U_1 \ \stackrel{\alpha_2}\longleftrightarrow \ \dotsc \ \stackrel{\alpha_n}\longleftrightarrow \ U_n,
 \]
 where $ \xleftrightarrow[]{\alpha_i}$ indicates that the arrow is allowed to assume either direction. The path $\Pi$ is \emph{closed} if $U_0 = U_n$, i.e. $0=n$ as objects of $\cI$. If $\Pi$ is closed, we denote by $\widehat\Pi$ the corresponding open path for which $0\neq n$, i.e., $U_0$ and $U_n$ are considered as distinct objects of the open path $\widehat\Pi$. Let
 \[
  \mu_\Pi: \ \lim\Pi \ \longrightarrow \ \lim\widehat\Pi
 \]
 be the morphism that is induced by the identity maps $U_i\to U_i$. 
  
 The diagram $\cU$ is \emph{monodromy free} if and only if $\mu_\Pi$ is an isomorphism for all closed paths $\Pi$ in $\cU$.
\end{rem}

\begin{ex}[Atlas of a covering family]\label{ex: atlas}
 The prototype of a monodromy free $\fs$-presentation stems from a family $\cS=\{U_i\to X\}$ of principal opens, assuming that all morphisms $U_i\to X$ are monomorphisms. We define the \emph{atlas of $\cS$} as the diagram $\cU_\cS$ that consists of the $U_i$ and all pairwise fibre products $U_{ij}=U_i\times_XU_j$ together with the canonical projections $U_{ij}\to U_i$ and $U_{ij}\to U_j$.

 The diagram $\cU_\cS$ consists of principal open immersions since the projections $U_{ij}\to U_i$ are the base change of the principal open $U_j\to X$ along $U_i\to X$. It is monodromy free by the alternative characterization of this property from \autoref{rem: intrinsic characterization of monodromy free diagrams}.
\end{ex}

Note that if $\cS=\{U_i\to X\}$ is a covering family in $\cT$, then the sheaf axiom for $\cS$ implies that the canonical morphism $\colim_{\Sh_\fs}\cU_\cS \to X$ is an isomorphism of sheaves on $\fs$. Thus $\fs$-presentations provide an adequate substitute for covering families, which leads to the following notion of an $\fs$-scheme.

\begin{df}
 An \emph{affine $\fs$-scheme} is a representable sheaf on $\fs$, i.e., it is isomorphic to $h_X=\Hom(-,X)$ for some $X$ in $\Aff$. An $\fs$-scheme is a sheaf $X$ on $\fs$ that is isomorphic to the colimit $\colim_{\Sh_\fs}\cU$ of a monodromy free $\fs$-presentation. We define $\Sch_\fs$ as the full subcategory of $\Sh_\fs$ that consists of all $\fs$-schemes.
\end{df}

\begin{rem}
 This definition of $\fs$-presentation deviates from that in \cite{toolkit}, where $\fs$-presentations are required to be closed under limits of finite connected subdiagrams but are in general not required to be monodromy free. This is relevant if one aims to cover a generalization of algebraic spaces and similar objects. Since we do not touch such generalizations in the present text, we can allow ourselves to work with the present simpler notion of $\fs$-presentations.
\end{rem}

%=======================================================
\subsection{Morphisms of \texorpdfstring{$\fs$}{s}-schemes}
\label{sub:MorphismOfSchemes}
%=======================================================

In this section, we explain in which sense a morphism of $\fs$-schemes is induced by a morphism of $\fs$-presentations.

\begin{df}
 Let $\mathcal U:\cI_\cU\to\Aff$ and $\mathcal V:\cI_\cV\to\Aff$ be $\fs$-presentations. 
 A \emph{morphism of $\fs$-presentations $\Phi: \mathcal V \rightarrow \mathcal U$} is a functor $\underline \Phi: \mathcal I_{\mathcal V} \rightarrow \mathcal I_{\mathcal U}$ together with a natural transformation $\hat{\Phi}: \mathcal V \rightarrow \mathcal U\circ \underline \Phi$. 
 Let $\Pres_\fs$ be the category of $\fs$-presentations.
\end{df}

A morphism $\Phi:\cV\to\cU$ of $\fs$-presentations induces a morphism between their colimits in $\Sh_\fs$. The Yoneda embedding $\Aff\to\Sh_\fs$ factorizes into
\[
 \Aff \ \longrightarrow \ \Pres_\fs \ \stackrel{\colim}\longrightarrow \ \Sch_\fs \ \longrightarrow \ \Sh_\fs,
\]
where the first functor sends $X$ to the $\fs$-presentation $\cU$ that consists of a single object~$X$. 
The first and the third functors are fully faithful. 
The second functor is not, but it is ``jointly full'' in the following sense.

\begin{thm}[{\cite{toolkit}}]\label{thm: morphisms of s-schemes are affinely presented}
 Let $\varphi:X\to Y$ be a morphism of $\fs$-schemes. Then there exists a morphism $\Phi:\cU\to\cV$ of $\fs$-presentations and isomorphisms $X\simeq\colim_{\Sh_\fs}\cU$ and $Y\simeq\colim_{\Sh_\fs}\cV$ such that the diagram
 \[
  \begin{tikzcd}[column sep=60, row sep=20]
   X \ar[r,"\varphi"] \ar[d,"\simeq"] & Y \ar[d,"\simeq"] \\
   \colim_{\Sh_\fs}\cU \ar[r,"\colim\Phi"] & \colim_{\Sh_\fs}\cV
  \end{tikzcd}
 \]
 commutes.
\end{thm}

\begin{proof}[Idea of proof]
 The proof is based on the notion of a refinement of $\fs$-presentations, which is a certain type of morphism $\Phi:\cW\to\cU$ that induces an isomorphism between the colimits of $\cW$ and $\cU$ in $\Sh_\fs$. 
 Refinements are put to work in the following way: one begins with an arbitrary choice of $\fs$-presentations $\cU$ and $\cV$ whose respective colimits in $\Sh_\fs$ are $X$ and $Y$. 
 Then one uses the inverse image of the affine $\fs$-schemes $\cV(j)$ under $\varphi:X\to Y$ to construct a suitable refinement $\cW\to\cU$ that allows for a morphism $\Phi:\cW\to \cV$ that represents $\varphi:X\to Y$.
\end{proof}

%=======================================================
\subsection{Smooth manifolds as \texorpdfstring{$\fs$}{s}-schemes}
\label{subsubsection: smooth manifolds as s-schemes}
%=======================================================

Let $\Aff=\Aff_{\cC^\infty}$ be the category of open subsets of Euclidean space and smooth maps, equipped with the Euclidean topology $\cT$, which defines the site $\fs=(\Aff,\cT)$. Let $\Mf_{\cC^\infty}$ be the category of smooth manifolds.

The Yoneda embedding identifies $\Aff$ with its essential image in $\Sh_\fs$, the category of affine $\fs$-schemes. This extends to an embedding $\yo:\Mf_{\cC^\infty}\to\Sh_\fs$ by sending a smooth manifold $X$ to the sheaf $\Hom(-,X)$ on $\fs$.

Since manifolds are by definition second countable, we need to make an analogous adaptation to $\Sch_\fs$. We define $\Sch_\fs^{\aleph_0}$ as the full subcategory of $\Sch_\fs$ that consists of colimits of countable $\fs$-presentations $\cU:\cI\to\Aff$ (i.e., the class of objects of $\cI$ is countable if not finite).

\begin{prop}\label{prop: smooth manifolds as s-schemes}
 The embedding $\yo:\Mf_{\cC^\infty}\to\Sh_\fs^{\aleph_0}$ has essential image $\Sch^{\aleph_0}_\fs$, and the restriction $\Mf_{\cC^\infty}\to\Sch_\fs^{\aleph_0}$ is an equivalence of categories.
\end{prop}

\begin{proof}[Proof outline]
 Let $X$ be a smooth manifold. Then $X$ has a countable open cover $\{\iota_i:U_i\to X\}_{i\in \cI}$ by open subsets $U_i$ of Euclidean space such that all $\iota_i$ are smooth. This means that the pairwise intersections $U_{ij}=U_{i}\times_XU_{j}$ are in~$\Aff$ and that the canonical projections $U_{ij}\to U_i$ are smooth open topological embeddings. The collection of all $U_i$ and $U_{ij}$ and all canonical projections $U_{ij}\to U_i$ is a diagram $\cU$ of principal opens. The proof that $\cU$ is monodromy free, and thus an $\fs$-presentation, relies on the intrinsic characterization of monodromy freeness from \autoref{rem: intrinsic characterization of monodromy free diagrams}.

 By the local nature of a sheaf on $\fs$, the sheaf $\Hom(-,X)$ is canonically isomorphic to $\colim_{\Sh_\fs} \cU$, which shows that $\Hom(-,X)$ is indeed an $\fs$-scheme and that the image of $\Mf_{\cC^\infty}\to\Sh_\fs$ is contained in $\Sch_\fs^{\aleph_0}$.

 Conversely, every $\fs$-scheme in $\Sch^{\aleph_0}_\fs$ is by definition the colimit of a countable $\fs$-presentation $\cU$. The monodromy freeness of $\cU$ implies that the colimit of $\cU$ in $\Mf_{\cC^\infty}$ is a smooth manifold $X$ and that the family $\{U_i\to X\}$ is an atlas for $X$. In conclusion, this shows that $\Sch_\fs$ is the essential image of $\yo:\Mf_{\cC^\infty}\to\Sh_\fs$.

 The proof that $\yo$ is fully faithful proceeds along the following arguments. Let $X$ and $Y$ be smooth manifolds with corresponding $\fs$-schemes $\yo_X$ and $\yo_Y$, and let $\varphi:\yo_X\to \yo_Y$ be a morphism of $\fs$-schemes. By \autoref{thm: morphisms of s-schemes are affinely presented}, $\varphi$ is the colimit of a morphism $\Phi:\cU\to\cV$ of $\fs$-presentations $\cU$ and $\cV$ with respective colimits $\yo_X$ and $\yo_Y$ in $\Sh_\fs$. Since the colimits of $\cU$ and $\cV$ in $\Mf_{\cC^\infty}$ are $X$ and $Y$, respectively, the colimit of $\Phi$ in $\Mf_{\cC^\infty}$ yields a smooth map $\psi:X\to Y$ with image $\yo(\psi)=\varphi$. Thus $\yo$ is full.

 Next we consider two morphisms $\varphi,\psi:X\to Y$ with $\yo (\varphi)=\yo (\psi)$. Choose smooth open coverings $\{U_i\to X\}$ and $\{V_j\to Y\}$ by open subsets $U_i$ and $V_j$ of Euclidean space. After a suitable refinement of $\{U_i\to X\}$, we can assume that $\varphi$ restricts to maps $\varphi:U_i\to V_{\underline\Phi(i)}$ and that $\psi$ restricts to maps $\psi_i:U_i\to V_{\underline\Psi(i)}$ for suitable maps $\underline\Phi$ and $\Psi$ between the indices of the $U_i$ and the $V_j$.

 Let $\cU$ be the $\fs$-presentation that corresponds to $\{U_i\to X\}$ and $\cV$ the $\fs$-presentation that corresponds to $\{V_j\to Y\}$, as constructed in the first paragraph of the proof. Then the maps $\varphi_i:U_i\to V_{\underline\Phi(i)}$ induce a morphism $\Phi:\cU\to\cV$ and the maps $\psi_i:U_i\to V_{\underline\Phi(i)}$ induce a morphism $\Psi:\cU\to\cV$. The colimits of $\Phi$ and $\Psi$ in $\Sh_\fs$ are equal to $\yo (\varphi)=\yo (\psi)$. Since the canonical inclusions $U_i\to X$ are monomorphisms, this means that $\varphi_i$ and $\psi_i$ agree as maps from $U_i$ to $V_{\Phi(i)}\cap V_{\Psi(i)}$. In consequence, $\Phi$ and $\Psi$ induce equal colimits $\varphi=\psi$ in $\Mf_{\cC^\infty}$, which shows that $\yo$ is faithful.
\end{proof}

%=======================================================
\subsection{Usual schemes as \texorpdfstring{$\fs$}{s}-schemes}
\label{subsubsection: usual schemes as s-schemes}
%=======================================================

Let $\Aff=\Aff_\Z$ be the category of affine schemes, equipped with the Zariski topology $\cT$ and $\fs=(\Aff,\cT)$. Let $\Sch_\Z$ be the usual category of schemes. Then the Yoneda embedding $\yo:\Aff\to\Sh_\fs$ extends to a functor $\yo:\Sch_\Z\to\Sh_\fs$, due to the local nature of scheme morphisms.

\begin{prop}\label{prop: schemes as s-schemes}
 The essential image of $\yo:\Sch_\Z\to\Sh_\fs$ is $\Sch_\fs$, and the restriction $\yo:\Sch_\Z\to\Sch_\fs$ is an equivalence of categories.
\end{prop}

\begin{proof}[Proof outline]
 It is a well-known fact that $\yo:\Sch_\Z\to\Sh_\fs$ is fully faithful (e.g. see \cite[Prop.\ IV-2]{eh06}). Thus it suffices to show that its essential image is $\Sch_\Z$. The proof is analogous to that of \autoref{prop: smooth manifolds as s-schemes} and as follows.

 We begin to show that the image of $\yo$ is contained in $\Sch_\fs$. Consider a scheme $X$ and choose an affine open covering $\{U_i\to X\}$. Then the pairwise intersections $U_{ij}=U_i\times_XU_j$ admit affine open coverings $\{U_{ijk}\to U_{ij}\}$. The union of all the open immersions $U_{ijk}\to U_i$ and $U_{ijk}\to U_j$ yields a diagram $\cU$, which is monodromy free since all canonical inclusions into $X$ are monomorphisms. By the local nature of morphisms of schemes, we find that $\yo_X=\colim_{\Sh_\fs}\cU$, which shows that $\yo_X$ is in $\Sch_\fs$.

 Conversely, every $\fs$-scheme $Z$ is isomorphic to the colimit of an $\fs$-presentation $\cU$. Since $\cU$ is monodromy free, its colimit $X$ as a topological space is the union of all affine schemes in $\cU$, and the structure of $X$ is determined by its restriction to the open subsets $U_i\subset X$, which turns $X$ into a scheme. Once again, the local nature of scheme morphisms implies that $Z=\yo_X$.
\end{proof}

%%%%%%%%%%%%%%%%%%%%%%%%%%%%%%%%%%%%%%%%%%%%%%%%%%%%%%%%%%%%%%%%%%%%%%%%%%%%%%%%%%%%%%%%%%%%%%%%%%%
%%%%%%%%%%%%%%%%%%%%%%%%%%%%%%%%%%%%%%%%%%%%%%%%%%%%%%%%%%%%%%%%%%%%%%%%%%%%%%%%%%%%%%%%%

%=======================================================
\section{Points for \texorpdfstring{$\fs$}{s}-schemes}
\label{section: PointsInSSchemes}
%=======================================================

In this section, we explain how the collection of open subschemes of an $\fs$-scheme determines a topological space as its Stone dual. In the cases of manifolds and schemes, the Stone dual agrees with the usual topological spaces.

%=======================================================
\subsection{Stone Duality}
\label{subsection: Stone duality}
%=======================================================

A \emph{locale} is a complete distributive lattice, which is a poset $\Lambda$ for which every subset $S\subset\Lambda$ has a largest lower bound $\bigwedge S$ (the \emph{meet of $S$}) and a least upper bound $\bigvee S$ (the \emph{join of $S$}) such that finite meets distribute over joins, i.e., $u\wedge(\bigvee S)=\bigvee_{v\in S} (u\wedge v)$. A locale $\Lambda$ has, in particular, a \emph{bottom}, or smallest element, $\0=\bigwedge \Lambda$ and a \emph{top}, or largest element, $\1=\bigvee\Lambda$. A \emph{morphism $f:\Lambda\to\Lambda'$ of locales} is a map $f^\ast:\Lambda'\to\Lambda$ in the opposite direction that preserves $\0$, $\1$, arbitrary joins and finite meets. We denote the category of locales by $\Loc$.

The prototype of a locale is the collection of open subsets of a topological space (where the meet of an arbitrary collection of open subsets is the open interior of their set-theoretic intersection), and a continuous map induces a morphism of locales. This defines a functor $(-)^\loc:\TopSp\to\Loc$ from topological spaces to locales.

For example, the locale of open subsets of the one point space $\bbone=\{\ast\}$ consists of the two open subsets $\0=\emptyset$ and $\1=\bbone$, i.e., the locale $\bbone^\loc$$=\{\0,\1\}$ is the Boolean lattice. Note that a topological space $X$ can be recovered as $\Hom_{\TopSp}(\bbone,X)$, equipped with the compact-open topology. 

The \emph{Stone dual} of a locale $\Lambda$ is the topological space $\Lambda^\top=\Hom_{\Loc}(\bbone^\loc,\Lambda)$, equipped with the topology that consists of open subsets of the form $\{f:\bbone^\loc\to\Lambda\mid f^\ast(u)$ $=\1\}$ with $u\in\Lambda$. The topology on $\Lambda^\top$ is functorial in $\Lambda$, which yields a functor $(-)^\top:\Loc\to\TopSp$.

\begin{thm}[Stone duality]
 The functor $(-)^\top:\Loc\to\TopSp$ is left adjoint to $(-)^\loc:\TopSp\to\Loc$, and this adjunction restricts to mutually inverse equivalences between the essential image of $(-)^\top$ (``sober'' topological spaces) and the essential image of $(-)^\loc$ (``spatial'' locales):
 \[
  \begin{tikzcd}[column sep=80,row sep=35]
   \TopSp \ar[r,"(-)^\loc","{\rotatebox[origin=c]{180}{$\perp$}}"',shift left=5pt] \ar[rd,->>,thin,start anchor=330,end anchor=150] & \Loc \ar[l,"(-)^\top",shift left=5pt] \ar[ld,->>,thin,crossing over,start anchor=210,end anchor=30] \\
   \big\{\text{sober topological spaces}\big\} \ar[r,"(-)^\loc","\simeq"',shift left=5pt] \ar[u,hook] & \big\{\text{spatial locales}\big\} \ar[l,"(-)^\top",shift left=5pt] \ar[u,hook]
  \end{tikzcd}
 \]
\end{thm}

\begin{rem}\label{rem: alternative characterizations of sober spaces and Stone duals}
 A topological space is sober if every closed irreducible subspace is the closure of a unique point. Examples are Hausdorff spaces and schemes. There is a similar explicit characterization of spatial locales. Since we do not need these characterizations in this text, we omit a discussion of this and refer to \cite[II.1]{Johnstone1986StoneSpaces}.
 
 The elements $f^\ast:\Lambda\to\bbone$ of the Stone dual $\Lambda^\top=\Hom(\bbone^\loc,\Lambda)$ can be characterized in terms of their kernels $\ker f^\ast=\{u\in\Lambda\mid f^\ast(u)=\0\}$, which are the \emph{principal prime ideals of $\Lambda$}, i.e., the subsets of $\Lambda$ of the form $\fp=\{u\in\Lambda\mid u\leq \bigvee \fp\}$ for which $\Lambda-\fp$ is non-empty and closed under finite meets. Equivalently, $\cF=\Lambda-\fp$ is a completely prime filter.
\end{rem}

%=======================================================
\subsection{Open subschemes}
\label{subsection: open subschemes}
%=======================================================

Heuristically speaking, an open subscheme is the ``union'' of principal opens. 
This is made rigorous in terms of the following definition (see \autoref{ex: atlas} for the notion of an atlas).

\begin{df}
 Let $X$ be an affine $\fs$-scheme. An \emph{open immersion into $X$} is an $\fs$-scheme of the form $U=\colim_{\Sh_\fs}\cU_\cS$, where $\cU_\cS$ is the atlas of a family $\cS=\{U_i\to X\}$ of principal opens, together with a canonical inclusion $U\to X$.

 Two open immersions $U\to X$ and $V\to X$ are isomorphic if there is an isomorphism $U\simeq V$ that commutes with the open immersions. An \emph{open subscheme of $X$} is the isomorphism class $u=[U\to X]$ of an open immersion $U\to X$.
\end{df}

In the definition of an open immersions, we do not assume $\cS$ to be a covering family in $\cT$. In fact, if $\cS$ is in $\cT$, then its colimit is isomorphic to $X$ itself, and thus the canonical inclusion $U\to X$ is an isomorphism in this case. This means that proper open subschemes arise from families $\cS$ that are \textit{not} in $\cT$.

For the rest of this section, we assume that the collection of open subschemes $u=[U\to X]$ of $X$ forms a set for every affine $\fs$-scheme $X$. We define $\Lambda_X$ as the set of open subschemes $u=[U\to X]$ of $X$, which we endow with the relation $[V\to X]\leq [U\to X]$ if and only if there is a morphism $V\to U$ of $\fs$-schemes.

\begin{prop}[{\cite{toolkit}}]\label{prop: locale of opens}
 The relation $\leq$ is a partial order that turns $\Lambda_X$ into a spatial locale, i.e., $\Lambda_X$ is equal to the locale of open subsets of its Stone dual $\Lambda_X^\top$.
\end{prop}

\begin{proof}[Proof outline]
 The composition law of morphisms implies that $\leq$ is a pre-order. The anti-symmetry of $\leq$ follows from the fact that an open immersion $U\to X$ is a monomorphism of $\fs$-schemes. The join of a subset $S$ of $\Lambda_X$ is given by the colimit of the ``union'' of suitable $\fs$-presentations of the open subschemes in $S$.

 Since every join-complete lattice is also meet-complete, this shows that $\Lambda_X$ is a locale. In particular, the pairwise meet is the represented by fibre product of two open subschemes, which agrees with the fibre product as sheaves on $\fs$. Thus the distributivity of the lattice follows from Giraud's axioms for topoi (fibre products commute with colimits).
 
 The proof that $\Lambda_X$ is spatial is not hard, but technical, based on an explicit characterization of spatial lattices. In essence, the argument consists of a reduction to ``stalks'' and the existence of maximal principal ideals in lattices.
\end{proof}

\begin{df}
 Let $X$ be an affine $\fs$-scheme and $\Lambda_X$ the locale defined as above. We call $\Lambda_X$ the \emph{locale of open subschemes of $X$} and define the \emph{underlying topological space of $X$} as $\underline X=\Lambda_X^\top$.
\end{df}

%=======================================================
\subsubsection{Points for manifolds}
%=======================================================

An open subset $X$ of Euclidean space is canonically homeomorphic to its underlying topological space, with respect to the Euclidean topology $\cT$ for $\Aff=\Aff_{\cC^\infty}$. This can be seen as follows.

Up to isomorphism, the principal opens of $X$ correspond to the open subsets of $X$. Since the open subsets of $X$ are closed under arbitrary unions, every element $u$ of $\Lambda_X$ is represented by a uniquely determined open subset $U$ of $X$, i.e., $u=[U\to X]$ for the tautological inclusion $U\to X$.

As explained in \autoref{rem: alternative characterizations of sober spaces and Stone duals}, the Stone dual of $\Lambda_X$ consists of the principal prime ideals of $\Lambda_X$. For $u=[U\to X]$, the principal ideal $\gen{u}$ of $\Lambda_X$ is prime if and only if $U\neq X$ and if $V\cap W\subset U$ (with $V,W\subset X$ open) implies that $V\subset U$ or $W\subset U$. This is the case if and only if $U=X-\{x\}$ for some $x \in X$. This association $\fp\mapsto x$ defines a bijection between the principal prime ideals $\fp\in\Lambda^\top_X$ and the points of $X$. It is evident that the open subset $\{f:\bbone^\loc\to\Lambda\mid  f^\ast(u)$ $ = \1 \}$  corresponds to $U$ for $u=[U\to X]$.

%=======================================================
\subsubsection{Points for schemes}
%=======================================================

In a similar fashion, the underlying topological space of an affine scheme $X=\Spec R$ is canonically homeomorphic to $\underline X$ in the case of $\Aff=\Aff_\Z$, equipped with the Zariski topology $\cT$.

Namely, the principal opens of $X$ are of the form $U_h=\Spec R[h^{-1}]$ with $h\in R$, and $\Lambda_X$ consists of all unions of principal opens, which are the open subschemes of $X$. As in the case of manifolds, a principal ideal $\gen{u}$ is generated by the isomorphism class $u=[U\to X]$ of an open immersion $U\to X$. The ideal $\gen{u}$ is prime if and only if the complement $Z$ of $U$ in $X$ is an irreducible closed subset, i.e., it cannot be covered by the union of two proper closed subsets. Since $X$ is a sober topological space, $Z$ has a unique generic point $x$. Thus the association $u\mapsto x$ defines a bijection between $\underline X$ and the underlying topological space of $X$. It is evident that the open subset $\{f:\bbone^\loc\to\Lambda\mid f^\ast(u)$ $= \1\}$ corresponds to $U$ for $u=[U\to X]$.

%=======================================================
\subsection{The structure sheaf}
%=======================================================

The underlying topological space $\underline X$ of an affine $\fs$-scheme $X$ can be endowed with a structure sheaf $\cO_X$, similar to the case of usual schemes or analytic spaces. For this, we assume that $\Aff$ is, in fact, defined as the opposite category of a category $\Alg$ of ``algebras,'' which itself could be defined purely formally as the opposite category of $\Aff$ (e.g.\ in the case of smooth manifolds). It might be more instructive, however, to consider the case of schemes as an example, for which $\Alg_\Z=\Aff_\Z^\op$ is the category of rings. We denote the anti-equivalences by $\Spec:\Alg\to\Aff$ and $\Gamma:\Aff\to\Alg$.

Consider an affine $\fs$-scheme $X$. By \autoref{prop: locale of opens}, the locale $\Lambda_X$ is spatial and thus its elements $u=[U\to X]$ correspond to the open subsets of the underlying topological space $\underline X$ of $X$. For an open subset $\underline U$ of $\underline X$ corresponding to $[U\to X]\in\Lambda_X$, we define $\Omega_{\underline U}$ as the category of all principal opens $V\to X$ of $X$ that factorize through $U\to X$ together with all inclusion maps $V\to V'$ between such principal opens. Note that an inclusion $\underline V\subset\underline U$ of open subsets of $\underline X$ induces an inclusion of diagrams $\Omega_{\underline V}\to\Omega_{\underline U}$.

The \emph{structure sheaf of $X$} is the presheaf $\cO_X$ on $\underline X$ in $\Alg$, given by
\[
 \cO_X(\underline U) \ = \ \lim{}_{\Alg} \ \Gamma\circ\Omega_{\underline U}
\]
for open subsets $\underline U$ of $\underline X$, together with the induced maps $\cO_X(\underline U)\to\cO_X(\underline V)$ for $\underline V\subset\underline U$.

\begin{prop}[\cite{toolkit}]\label{prop: structure sheaf}
 The presheaf $\cO_X$ is indeed a sheaf on the topological space $\underline X$. If $\underline U$ corresponds to a principal open $U\to X$, then $\cO_X(\underline U)=\Gamma U$.
\end{prop}

Note that, in particular, this implies that $\cO_X(\underline X)=\Gamma X$ recovers the ``global sections'' of $X$.

\medskip

This perspective on affine $\fs$-schemes can be extended to all $\fs$-schemes in the sense that every $\fs$-scheme naturally corresponds to a topological space together with a structure sheaf in $\Alg$. This can be conveniently formulated using $\fs$-presentations and their colimits in a suitable category of topological spaces together with a sheaf in $\Alg$. The definition of ``local'' morphisms in this level of generality would lead us too far astray, so we omit a presentation of this theory in this text.

%%%%%%%%%%%%%%%%%%%%%%%%%%%%%%%%%%%%%%%%%%%%%%%%%%%%%%%%%%%%%%%%%%%%%%%%%%%%%%%%%%%%%%%%%%%%%%%%%%%%%%%%%%%%%%%%%%%%%%%%%%%%%%%%%%%%%%%%%%%%%%%%%%%%%%%%
%%%%%%%%%%%%%%%%%%%%%%%%%%%%%%%%%%%%%%%%%%%%%%%%%%%%%%%%%%%%%%%%%%%%%%%%%%%%%%%%%%%%%%%%%%%%%%%%%%%%%%%%%%%%%%%%%%%%%%%%%%%%%%%%%%%%%%%%%%%%%%%%%%%%%%%%

%=======================================================
\section{Further examples}
\label{section: further examples}
%=======================================================

%=======================================================
\subsection{Topological spaces}
\label{subsection: topological spaces}
%=======================================================

Topological spaces can be realized as $\fs$-schemes for a certain site $\fs$. This case is somewhat particular since every $\fs$-scheme turns out to be affine. 

To explain, let $\Aff$ be the category $\TopSp$ of topological spaces and $\cP$ be the class of open topological embeddings. For the same reason as for manifolds, a family $\{\iota_i:U_i\to X\}$ of open topological embeddings is subcanonical if and only if it is jointly surjective. This defines the Grothendieck pretopology $\cT=\gen\cP_\can$ on $\TopSp$ and the site $\fs=(\TopSp,\cT)$. 

Since $\TopSp$ is closed under colimits, it turns out that $\Sch_\fs$ agrees with the essential image of the Yoneda embedding $\TopSp\to\Sh_\fs$, i.e., $\Sch_\fs$ is equivalent to $\TopSp$. 

Finally note that the underlying space $\underline X$ of an $\fs$-scheme $X$ is canonically homeomorphic to the topological space $X$, for similar reasons as this is the case for smooth manifolds.

%=======================================================
\subsection{Various types of manifolds}
\label{subsection: various types of manifolds}
%=======================================================

With the obvious adaptations to $\Aff$, the example of smooth manifolds as $\fs$-schemes generalizes to other types of manifolds, such as topological manifolds, $\cC^i$-manifolds and complex manifolds.

More concretely, in the case of $\cC^i$-manifolds, we define $\Aff$ as the category of open subsets of Euclidean space together with all $\cC^i$-maps. Principal opens are open topological embeddings that are in~$\cC^i$.

In the case of complex manifolds, $\Aff$ is the category of open subsets of $\C^n$ together with holomorphic maps. Principal opens are holomorphic open topological embeddings.

In each case, the class $\cP$ of principal opens defines a subcanonically saturated Grothendieck pretopology $\cT=\gen{\cP}_\can$ on $\Aff$ and a site $\fs=(\Aff,\cT)$ that yields a category $\Sch_\fs$ of $\fs$-schemes.

The same proof as in the case of smooth manifolds shows that the corresponding category of ($\cC^i$ or holomorphic) manifolds is equivalent to the category $\Sch_\fs^{\aleph_0}$ of $\fs$-schemes that are the colimit of a countable $\fs$-presentation. Similarly, the underlying topological space $\underline X$ of an affine $\fs$-schemes $X$ agrees in all cases with $X$, considered as an open subspace of Euclidean space.

%=======================================================
\subsection{Simplicial complexes}
%=======================================================

Simplicial complexes and simplicial maps form a category of $\fs$-schemes for a site $\fs$ with a Grothendieck pretopology $\cT$ that is not subcanonically saturated. This contrasts the other examples of this text.

The category of simplicial complexes with simplicial maps is equivalent to the category $\ASC$ of abstract simplicial complexes. To recall, an abstract simplicial complex is a set $C$ together with a family $\cX$ of finite subsets of $C$ that is closed under taking subsets and such that $C$ is the union of all subsets in $\cX$. The geometric interpretation of the non-empty finite subsets $A$ in $\cX$ is as $n$-simplex where $n=\# A-1$. The face relation of simplices is given by inclusions $B\subset A$ of non-empty subsets. A morphism between abstract simplicial complexes $(C_1,\cX_1)$ and $(C_2,\cX_2)$ is a map $F:C_1\to C_2$ that takes the subsets in $\cX_1$ to the subsets in $\cX_2$. 

This category can be realized as $\fs$-schemes in the following way. Let $\Aff$ be the category of finite sets and maps. A non-empty finite set $A$ corresponds to an $n$-simplex of dimension $n=\#A-1$.

A \emph{principal open} is an injection $B\hookrightarrow A$, which corresponds to a face map of simplices, provided that $B$ is non-empty. In order to guarantee that the class $\cP$ of principal opens is stable under base change, we need to include the injection $\emptyset\to A$ in $\cP$. 

We define $\cT$ as the Grothendieck pretopology on $\Aff$ that consists of all families $\cS=\{\iota_i:B_i\to A\}$ of principal opens that contain a bijection $\iota_i$ for some $i$. This Grothendieck pretopology is not subcanonically generated, as explained in \autoref{rem: subcanonically saturated Grothendieck topology for finite sets}.

This defines the site $\fs=(\Aff,\cT)$ and yields the Yoneda embedding $\Delta_+\to \Sh_\fs$. The functor $\Hom(-,X):\Delta_+\to\Sets$ defines a functor $\cF$ from the category $\ASC$ of abstract simplicial complexes to $\Sh_\fs$.

\begin{prop}
 The functor $\cF$ restricts to an equivalence of categories $\cF:\ASC\to\Sch_\fs$.
\end{prop}

This result can be proven in the same way as the corresponding result \autoref{prop: smooth manifolds as s-schemes} for smooth manifolds. There is, however, one significant simplification in this case: the image of an affine open $U\to X$ of an $\fs$-scheme $X$ under a morphism $\varphi:X\to Y$ is again an affine open. This means that the functor $\colim_{\Sh_\fs}:\Pres_\fs\to\Sh_\fs$ is full, which has the consequence that we do not need to construct suitable refinements in several steps of the proof.

Finally we like to mention that the underlying topological space $\underline A$ of an $n$-simplex $A$ stays in bijection with the faces of $A$, and that a face $B$ is contained in a face $C$ if and only if the point of $\underline A$ corresponding to $C$ is contained in the closure of the point corresponding to $B$. In particular, the point corresponding to the face $A\to A$ is the unique closed point of $\underline A$.  

\begin{rem}\label{rem: subcanonically saturated Grothendieck topology for finite sets}
 The Grothendieck pretopology $\cT$ from this section is not subcanonically saturated, i.e., $\cT$ is properly contained in $\cT'=\gen{\cP}_\can$, which consists of all jointly surjective families $\{\iota_i:B_i\to A\}$ of principal opens, or injective maps, $\iota_i:B_i\to A$.

 Indeed, a jointly surjective family $\{\iota_i:B_i\to A\}$ of principal opens is subcanonical since every map from $A$ into another set is uniquely determined by the images of the elements of $A$.  Conversely every subcanonical family  $\{\iota_i:B_i\to A\}$ of principal opens has to be jointly surjective since we could otherwise pull back such a family to a singleton contained in the complement of the images of the $\iota_i$ in $A$, which yields a covering family of the singleton by empty sets. This latter covering is not subcanonical and provides the desired contradiction.

 It turns out that the category of $\fs'$-schemes for the site $\fs'=(\Aff,\cT')$ is equivalent to $\Sets$ and not to $\ASC$. In this case, the underlying topological space $\underline A$ of a finite set is discrete and in canonical bijection with $A$ itself.
\end{rem}

%%%%%%%%%%%%%%%%%%%%%%%%%%%%%%%%%%%%%%%%%%%%%%%%%%%%%%%%%%%%%%%%%%%%%%%%%%%%%%%%%%%%%%%%%%%%%%%%%%%%%%%%%%%%%%%%%%%%%%%%%%%%%%%%%%%%%%%%%%%%%%%%%%%%%%%%
%%%%%%%%%%%%%%%%%%%%%%%%%%%%%%%%%%%%%%%%%%%%%%%%%%%%%%%%%%%%%%%%%%%%%%%%%%%%%%%%%%%%%%%%%%%%%%%%%%%%%%%%%%%%%%%%%%%%%%%%%%%%%%%%%%%%%%%%%%%%%%%%%%%%%%%%

%=======================================================
\section{Semiring schemes}
\label{section: semiring schemes}
%=======================================================

In this final part of the text, we employ the theory of $\fs$-schemes to define semiring schemes.

%=======================================================
\subsection{Semirings}
%=======================================================

In this text, a semiring is commutative with $0$ and $1$ and semiring morphisms preserve both $0$ and $1$. We denote the category of semirings by $\SRings$. The semiring $\N$ of natural numbers is an initial object in $\SRings$.

Let $R$ be a semiring. A \emph{finite localization of $R$} is a semiring morphism $f:R\to R'$ for which there are an element $h\in R$ and an isomorphism $\iota:R[h^{-1}]\to R'$ such that $f$ is the composition of the localization $\lambda_h:R\to R[h^{-1}]$ with $\iota$ (cf.\ \cite[Section 1.1]{GKMX} for localizations of semirings).

%=======================================================
\subsection{The Zariski site for semirings}
%=======================================================

Let $\Aff_\N$ the opposite category of $\SRings$ and $\Spec:\SRings\to \Aff_\N$ and $\Gamma:\Aff_\N\to\SRings$ be the anti-equivalences. A \emph{principal open} is a morphism $\varphi:X\to Y$ in $\Aff$ for which $\Gamma\varphi:\Gamma Y\to \Gamma X$ is a finite localization. Let $\cP$ be the class of all principal opens and $\cT=\gen{\cP}_\can$ the subcanonically saturated Grothendieck pretopology for $\Aff$ generated by $\cP$, which defines the site $\fs=(\Aff_\N,\cT)$.

\begin{df}
 An \emph{(affine) semiring scheme} is an (affine) $\fs$-scheme. We denote the category of semiring schemes by $\Sch_\N=\Sch_\fs$.
\end{df}

%=======================================================
\subsection{The two comparison theorems}
%=======================================================

Similar to usual schemes, the definition of semiring schemes in this text coincides with the notion of semischemes from \cite{GKMX}, in the sense of the two comparison theorems from \cite{L26}. Before we formulate these results, we provide a brief review of semischemes.

Let $R$ be a semiring. An \emph{ideal of $R$} is an additive subsemigroup $I$ of $R$ that is stable under the multiplication by $R$. A \emph{prime ideal of $R$} is an ideal $\fp$ of $R$ whose complement $S=R-\fp$ is a \emph{multiplicative subset}, i.e., $1\in S$ and $S$ is closed under multiplication.

The \emph{prime spectrum of $R$} is the set $\PSpec R$ of all prime ideals of $R$, endowed with the topology generated by the open subsets of the form
\[
 U_h^P \ = \ \big\{ \fp\in\PSpec R \, \big| \, h\notin \fp \big\}
\]
with $h\in R$.

\begin{thm}[{\cite[Thm.\ A]{L26}}]\label{thm: comparison1}
 Let $X=\Spec R$ be an affine semiring scheme. Then $\underline X$ is canonically homeomorphic to the prime spectrum of $R$.
\end{thm}

Semischemes are defined as certain types of \emph{semiringed spaces}, which are topological spaces $X$ equipped with a sheaf $\cO_X$ in semirings. A \emph{local morphism of semiringed spaces} is a continuous map $\varphi:X\to Y$ together with a morphism $\varphi^\#:\varphi^\ast\cO_Y\to\cO_X$ of sheaves on $X$ such that the induced morphism of stalks $\varphi_x:\cO_{Y,\varphi(x)}\to\cO_{X,x}$ maps non-units to non-units for every $x\in X$.

The prime spectrum $X=\PSpec R$ of a semiring $R$ is a semiringed space with respect to its \emph{structure sheaf $\cO_X$}, which is the sheaf on $X$ with values in $\SRings$ that is characterized by $\cO_X(U^P_h)=R[h^{-1}]$ (thus, in particular, $\cO_X(X)=R$) and the requirement that the restriction map $\cO_X(U^P_h)\to\cO_X(U^P_{gh})$ is the localization of $R[h^{-1}]$ at $g$ for all $g,h\in R$. Since the $U^P_h$ form a basis for the topology of $\PSpec R$ (in particular, $U^P_g\cap U^P_h=U^P_{gh}$), the sheaf $\cO_X$ is uniquely determined by this requirement. That $\cO_X$ exists is significantly harder to prove than for usual schemes; see \cite[Thm.\ 3.18]{GKMX}.

The structure sheaf $\cO_X$ gives the prime spectrum $X=\PSpec R$ of a semiring $R$ the structure of a semiringed space. A semiring morphism $f:R\to R'$ induces a local morphism $f^\ast:\PSpec R'\to \PSpec R$ between prime spectra that maps a prime ideal $\fp$ of $R'$ to $f^{-1}(\fp)$ and whose induced map between global sections is $f$ itself. This defines a contravariant fully faithful embedding of $\SRings$ into the category of semiringed spaces, cf.\ \cite[Prop.\ 3.22]{GKMX}.

An \emph{affine semischeme} is a semiringed space that is isomorphic to the prime spectrum of a semiring. Every open subset $U$ of a semiringed space $X$ is a semiringed space with respect to the restriction of $\cO_X$ to $U$. A \emph{semischeme} is a semiringed space $X$ that has a covering by affine semiring schemes. A morphism of (affine) semischemes is a local morphism of semiringed spaces. We denote the category of affine semischemes as $\SAff$ and the category of semiring schemes as $\SSch$.

The composition of $\Gamma:\Aff_\N\to\SRings$ with $\PSpec:\SRings\to \SAff$ defines an equivalence of categories $\eta:\Aff_\N\to\SAff$.

\begin{thm}[{\cite[Thm.\ B]{L26}}]\label{thm: comparison2}
 The functor $\eta:\Aff_\N\to\SAff$ extends to an equivalence $\hat\eta:\Sch_\N\to\SSch$ that is uniquely determined up to unique isomorphism by the property that $\hat\eta\big(\colim_{\Sch_\N}\cU\big)\simeq \colim_{\SSch}\eta\circ\cU$ for every $\fs$-presentation $\cU$.
\end{thm}

%=======================================================
\subsection{Visualizations}
%=======================================================

The various characterizations of a prime ideal of a ring lead to diverging concepts in the generality of semirings, which means that different properties of ideals are captured by different quantities in the world of semirings and semiring schemes.

\autoref{thm: comparison1} shows that the prime ideals of a semiring reflect the most flexible way to glue prime spectra along principal opens \emph{without losing information}. It turns out that virtually all other interesting aspects of prime ideals behaves well with these gluing conditions, which leads to certain secondary topological spaces attached to a semiring scheme $X$, which we call visualizations of $X$, a term that is inspired by the introduction of \cite{fujiwara2017foundationsrigidgeometryi}.

The concise definition of a visualization is as follows. \autoref{thm: comparison2} provides us with a functor $\underline{(-)}:\Sch_\N\to\TopSp$ that sends a semiring scheme $X$ to the underlying topological space of the corresponding semischeme $\hat\eta(X)$, which extends the previously constructed functor $\underline X=\Lambda^\top_X$ for affine semiring schemes $X$ by \autoref{thm: comparison1}.

\begin{df}
 A \emph{visualization of $\Sch_\N$} is a functor $\cV:\Sch_\N\to\TopSp$ together with a natural transformation $\upsilon:\cV\to\underline{(-)}$. 
\end{df}

A visualization is \emph{compatible with $\fs$-presentations} if it satisfies that $\cV(\colim_{\Sch_\N}\cU)\simeq \colim_{\TopSp}\cV\circ\cU$ (functorially in $\cU$) for all $\fs$-presentation~$\cU$.

Let $\cV':\Aff\to\TopSp$ be a functor and $\upsilon':\cV'\to\underline{(-)}$ a natural transformation. We say that $\cV$ is an \emph{extension of $\cV'$ (with respect to $\upsilon'$)} if the restriction of $\cV$ is isomorphic to $\cV'$ in a way such that the restriction of $\upsilon$ to $\Aff$ coincides with $\upsilon'$. Note that up to unique isomorphism, there is at most one extension $\cV$ of $\cV'$ that is compatible with $\fs$-presentations.

In the following, we discuss some visualizations of semiring schemes.

%=======================================================
\subsubsection{The kernel space}
%=======================================================

Let $R$ be a semiring. A \emph{$k$-ideal} is an ideal $I$ of $R$ that contains $a$ whenever it contains $a+b$ for some $b\in I$. The class of $k$-ideals corresponds to the class of kernels $\ker f=f^{-1}(0)$ of semiring morphisms $f:R\to R'$. In particular, the quotient map $R\to R/I$ has kernel $I$ if and only if $I$ is a $k$-ideal. For details, cf.\ \cite[Section 1.3]{GKMX} or \cite{Golan1999}.

We define the \emph{$k$-spectrum of $R$} as the subspace $\PSpec^kR$ of $\PSpec R$ that consists of all prime $k$-ideals of $R$, which comes together with the inclusion map $\iota^k:\PSpec^kR\to\PSpec R$. The \emph{kernel space of a semiring scheme $X$} is the space $X^k$ defined by the following result.

\begin{prop}\label{prop: kernel space of a semiring scheme}
 There is a unique extension of $\PSpec^k\circ\Gamma:\Aff_\N\to\TopSp$ (with respect to $\iota^k$) to a visualization $(-)^k:\Sch_\N\to\TopSp$ that is compatible with $\fs$-presentations.
\end{prop}

\begin{proof}[Proof outline]
 The extension is unique since every morphism of semiring schemes is the colimit of a morphism of $\fs$-presentations by \autoref{thm: morphisms of s-schemes are affinely presented}. By \autoref{thm: comparison2}, the equivalence $\hat\eta:\Sch_\N\to\SSch$ is compatible with $\fs$-presentations, 
 which reduces the proof to verifying that the subset of $k$-ideals is independent of affine open covering of a semischeme $X$. This is the case since the subset of $k$-ideals of the prime spectrum commutes with localizations; cf.\ \cite[Lemma 2.16]{GKMX}.
\end{proof}

%=======================================================
\subsubsection{Congruences}
%=======================================================

Let $R$ be a semiring. A \emph{congruence on $R$} is an equivalence relation $\fc$ on $R$ that is additive and multiplicative in the sense that for all $(a,b)\in\fc$ and $c\in R$ also $(a+c,b+c)$ and $(ac,bc)$ are in~$\fc$. Different notions of primality for prime congruences can be found in the literature.

We say that a congruence $\fc$ on $R$ is
\begin{itemize}
 \item \emph{weak prime} if $(ab,0)\in \fc$ implies that $(a,0)\in \fc$ or $(b,0)\in \fc$ (\cite{Lescot2012});
 \item \emph{strong prime} if $(ab,ac)\in \fc$ implies that $(a,0)\in \fc$ or $(b,c)\in \fc$ (\cite{Lorscheid2012});
 \item \emph{twisted prime} if $(ac+bd,ad+bc)\in\fc$ implies that $(a,b)\in \fc$ or $(c,d)\in \fc$ (\cite{JooMincheva2018}).
\end{itemize}

\begin{lemma}
 A twisted prime congruence is a strong prime congruence. A strong prime congruence is a weak prime congruence. If $\fc$ is a weak prime congruence, then $I_\fc=\{a\in R\mid (a,0)\in\fc\}$ is a prime $k$-ideal.
\end{lemma}

\begin{proof}
 Let $\fc$ be a twisted prime congruence and consider $(ab,ac)\in \fc$. Since $(ab,ac)=(ab+0c,ac+0b)$, we have $(a,0)\in R$ or $(b,c)\in R$, which shows that $\fc$ is strong prime.

 Let $\fc$ be a strong prime congruence and consider $(ab,0)\in R$. Since $(ab,0)=(ab,0a)$, we have $(a,0)\in R$ or $(b,0)\in R$, which shows that $\fc$ is weak prime.

 Let $\fc$ be a weak prime congruence and consider $ab\in I_\fc$, i.e., $(ab,0)\in\fc$. Then $(a,0)\in\fc$ or $(b,0)\in\fc$ and thus $a\in I_\fc$ or $b\in I_\fc$, which shows that $I_\fc$ is prime. It is a $k$-ideal since if it contains $b$ and $a+b$, i.e., $(b,0),(a+b,0)\in\fc$, then by symmetry and additivity, also $(a,a+b)\in\fc$ and thus by transitivity $(a,0)\in\fc$, i.e., $a\in I_\fc$.
\end{proof}

Each of these three notions of prime congruences defines a visualization of $\Sch_\N$. Let $X=\Spec R$ be an affine semiring scheme. We define the \emph{weak congruence spectrum of $R$} as the set  $X^w=\Cong^w R$ of all weak prime congruences on $R$ together with the topology generated by open subsets of the form
\[
 U_{a,b}^w \ = \ \big\{ \fc\in X^w \, \big| \, (a,b)\notin\fc \big\}
\]
with $a,b\in R$. The map $\iota^w:X^w\to \underline X$ given by $\fc\mapsto I_\fc$ is continuous because the inverse image of $U^P_h$ is $U^w_{h,0}$. We define the \emph{strong congruence spectrum of $R$} as the subspace $X^s=\Cong^s R$ of strong prime congruences in $X^w$. 
We define the \emph{twisted congruence spectrum of $R$} as the subspace $X^t=\Cong^t R$ of twisted prime congruences in $X^w$. The respective compositions of $\iota^w$ with the embeddings into $X^w$ yield continuous maps $\iota^s:X^s\to\underline X$ and $\iota^t:X^t\to\underline X$.

The \emph{weak, strong and twisted prime congruence space of $X$} is defined in terms of the following proposition.

\begin{prop}\label{prop: weak prime congruence spaces}
 Let $x\in\{w,s,t\}$. There is a unique extension of $(-)^x$ with respect to $\iota^x$ to a visualization $(-)^x:\Sch_\N\to\TopSp$ that is compatible with $\fs$-presentations.
\end{prop}

\begin{proof}[Proof outline]
 The extension is unique since every morphism of semiring schemes is the colimit of a morphism of $\fs$-presentations by \autoref{thm: morphisms of s-schemes are affinely presented}. By \autoref{thm: comparison2}, the equivalence $\hat\eta:\Sch_\N\to\SSch$ is compatible with $\fs$-presentations, which reduces the proof to verifying that the congruence spectrum $\Cong^x R[h^{-1}]$ of a localization of $R$ is equal to the open subset $U^x_{h,0}=U^w_{h,0}\cap\Cong^x R$ of $\Cong^x R$, which can be proven in a similar way as in the case of prime ideals; cf.\ \cite[Lemma 2.16]{GKMX}.
\end{proof}

\begin{cor}
 Every $\fs$-scheme comes with a sequence
 \[
  \begin{tikzcd}
   X^t \ar[r,hook] & X^s \ar[r,hook] & X^w \ar[r,->>] & X^k \ar[r,hook] & \underline X
  \end{tikzcd}
 \]
 of continuous maps that is functorial in $X$. \qed
\end{cor}

%\pagebreak[2]

%%%%%%%%%%%%%%%%%%%%%%%%%%%%%%%%%%%%%%%%%%%%%%%%%%%%%%%%%%%%%%%%%%%%%%%%%%%%%%%%%%%%%%%%%%%%%%%%%%%%%%%%%%%%%%%%%%%%%%%%
%%%%%%%%%%%%%%%%%%%%%%%%%%%%%%%%%%%%%%%%%%%%%%%%%%%%%%%%%%%%%%%%%%%%%%%%%%%%%%%%%%%%%%%%%%%%%%%%%%%%%%%%%%%%%%%%%%%%%%%%
\begin{small}
 \bibliographystyle{alpha}
 \bibliography{ref}
\end{small}
%%%%%%%%%%%%%%%%%%%%%%%%%%%%%%%%%%%%%%%%%%%%%%%%%%%%%%%%%%%%%%%%%%%%%%%%%%%%%%%%%%%%%%%%%%%%%%%%%%%%%%%%%%%%%%%%%%%%%%%%
%%%%%%%%%%%%%%%%%%%%%%%%%%%%%%%%%%%%%%%%%%%%%%%%%%%%%%%%%%%%%%%%%%%%%%%%%%%%%%%%%%%%%%%%%%%%%%%%%%%%%%%%%%%%%%%%%%%%%%%%
 
\end{document}